\documentclass{article}
\usepackage{amsthm,amsmath}
\usepackage{amssymb}
\usepackage{xcolor}
\usepackage{tikz}
\usetikzlibrary{shadings}

\newtheorem{theorem}{Theorem}

\newtheorem{proposition}[theorem]{Proposition}

\newcommand{\size}[1]{\left|#1\right|}
\newcommand\Set[2] {\left\{{#1}:\,{#2}\right\}}
\newcommand\Setx[1] {\left\{{#1}\right\}}

\begin{document}{}
\title{\textbf{A note on vertex-critical\\induced subgraphs of shift graphs}}\author{Tom\'{a}\v{s} Kaiser$^{\:1}$
  \and Mat\v{e}j Stehl\'{\i}k$^{\:2}$
  \and Riste \v{S}krekovski$^{\:3}$
  \and Bjarne Toft$^{\:4}$}

\date{}

\maketitle

\begin{abstract}
  Shift graphs, introduced by Erd\H{o}s and Hajnal in 1964, form one of the simplest known non-recursive constructions of triangle-free graphs with arbitrarily large chromatic number. In this note, we identify a surprising property: for each integer $k \geq 1$, the smallest $k$-chromatic shift graph contains a \emph{unique} induced $k$-vertex-critical subgraph. We give an explicit description of this subgraph and prove its uniqueness. This provides a new family of vertex-critical triangle-free graphs of arbitrarily large chromatic number.
\end{abstract}
\footnotetext[1]{Department of Mathematics and European Centre of
 Excellence NTIS (New Technologies for the Information Society),
 University of West Bohemia, Pilsen, Czech Republic. E-mail: \texttt
 {kaisert@kma.zcu.cz}.}%
\footnotetext[2]{Université Paris Cité, CNRS, IRIF, F-75006, Paris,
 France. E-mail: \texttt{matej@irif.fr}.  Partially supported by ANR
 project DISTANCIA(ANR-17-CE40-0015).}%
\footnotetext[3]{Faculty of Mathematics and Physics, University of
 Ljubljana; Faculty of Information Studies in Novo Mesto;
 Rudolfovo -- Science and Technology Centre, Novo Mesto, Slovenia.
 E-mail: \texttt{skrekovski@gmail.com}. Partially supported by
 Slovenian Research and Innovation Agency ARIS program P1-0383,
 project J1-3002, and the annual work program of Rudolfovo.}%
\footnotetext[4]{Department of Mathematics and Computer Science, University of Southern Denmark, Campusvej 55, DK-5230 Odense M, Denmark.
 E-mail: \texttt{btoft@imada.sdu.dk}.}%

\section{Introduction}

The study of graphs that are simultaneously highly chromatic and sparse---typically having large (odd) girth---has had a profound influence on graph theory. It has played a central role in the development of both the probabilistic method~\cite{Erd59} and the topological method~\cite{Lov78}. For background on this rich line of research, see the surveys~\cite{JT95,Nes13,SS20,SST24,Tof95}.

The earliest constructions of triangle-free graphs with arbitrarily large chromatic number, due to Tutte (writing under an alias)~\cite{Des47,Des54}, Zykov~\cite{Zyk49}, Kelly and Kelly~\cite{KK54}, and Mycielski~\cite{Myc55}, were all recursive. The first non-recursive constructions were shift graphs~\cite{EH64}, which are the focus of this note, and Kneser graphs~\cite{Lov78}.

For $N \ge 5$, the \emph{shift graph} $G_{N,2}$ is defined as follows. The vertices are ordered pairs $(x,y)$ with integers $1 \le x < y \le N$, and two vertices $(x,y)$ and $(y,z)$ are adjacent whenever $x < y < z$. Erdős and Hajnal~\cite{EH64} and, independently, Harner and Entringer~\cite{HE72} proved that $G_{N,2}$ is triangle-free and
\[
  \chi(G_{N,2}) = \lceil \log_2 N \rceil.
\]
See also~\cite[Problem 9.26]{Lov93} and~\cite{ST11}.

A vertex of a graph is \emph{critical} if its removal decreases the chromatic number. A graph is \emph{$k$-vertex-critical} if it is $k$-chromatic and all of its vertices are critical. Vertex-critical graphs are central objects in graph colouring, as structural questions about all $k$-chromatic graphs often reduce to the vertex-critical case.

Our main result shows that each shift graph $G_{2^n+1,2}$ contains a \emph{unique} induced $(n+1)$-vertex-critical subgraph, which we describe explicitly. This phenomenon is in sharp contrast with Kneser graphs $\mathrm{KG}(n,k)$, which contain $(n-1)!/2$ isomorphic copies of the vertex-critical Schrijver graph $\mathrm{SG}(n,k)$~\cite{Sch78}, along with additional non-isomorphic vertex-critical subgraphs.

For integers $p \le q$, let $[p,q]$ denote the interval $\{p,p+1,\dots,q\}$.  
For $n \ge 2$, define intervals $I_0,\dots,I_n \subseteq [1,2^n+1]$ by
\[
I_\ell = [\,2^\ell,\; 2^n - 2^{n-\ell} + 2\,], \qquad 0 \le \ell \le n.
\]
Thus the interval $[1,2^n+1]$ contains $2^\ell - 1$ elements to the left of $I_\ell$ and $2^{n-\ell} - 1$ elements to the right of $I_\ell$. Let
\[
W = \{(x,y) : 1 \le x < y \le 2^n+1 \text{ and } \{x,y\} \subseteq I_\ell \text{ for some }\ell\}.
\]
A geometric depiction of $W$ for $n=4$ is given in Figure~\ref{fig:diagram}.

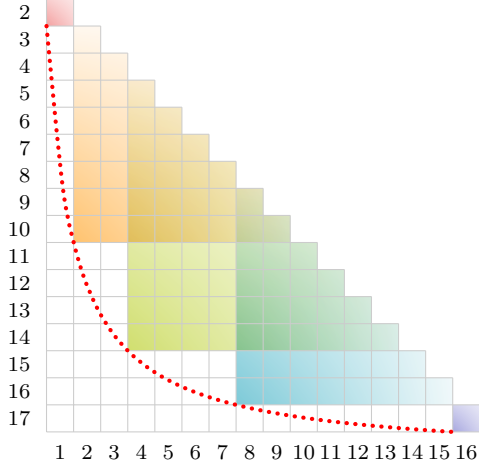
\begin{figure}
\centering
\begin{tikzpicture}[scale=0.36]
\def\n{4}
\pgfmathtruncatemacro\k{2^\n+1}
\pgfmathtruncatemacro\m{\n-1}
\pgfmathtruncatemacro\l{\k-1}

\definecolor{c0}{RGB}{240,120,120}  
\definecolor{c1}{RGB}{255,170,50}   
\definecolor{c2}{RGB}{190,210,50}   
\definecolor{c3}{RGB}{70,180,200}   
\definecolor{c4}{RGB}{130,130,210}  

\begin{scope}[blend group = multiply]
\foreach \i [evaluate=\i as \x using 2^\i,
    evaluate=\i as \xnext using {min(2^(\i+1),\l)},
    evaluate=\i as \y using 2^(\n-\i),
    evaluate=\i as \ynext using 2^(\n-\i-1),
    evaluate=\i as \top using 2^\n-2^\i+2] in {0,...,\n} {
        \shade[lower left=c\i!100!black,
         upper right=c\i!5!white,
         opacity=0.7] (\x,\y) -- (\xnext,\y)
        \foreach \i in {\y,...,\top} {
            -- (\k-\i+2, \i) -- (\k-\i+1, \i) -- (\k-\i+1, \i+1)
        }
    -- cycle;
}
\end{scope}

\foreach \i in {2,...,\k} {
    \draw[black!20] (1,\i) -- (\k-\i+2,\i)
    (\i,\k-\i+2) -- (\i,1);
}

\draw[black!20] (1,1) -- (\k,1)
    (1,\k) -- (1,1);

\foreach \i in {1,...,\l} {
    \node at (\i+0.5,0.25) {\footnotesize $\i$};
}
\foreach \i in {2,...,\k} {
    \node[anchor=east] at (0.75,\k-\i+1.5) {\footnotesize $\i$};
}

\draw[scale=1,domain=1:\k-1,smooth, variable=\x,red, ultra thick, line cap=round, dash pattern=on 0pt off 2\pgflinewidth] plot ({\x},{(\k-1)/(\x)});
\end{tikzpicture}
\caption{A geometric representation of the vertex set $W$ of the shift graph $G_{17,2}$. Each box with coordinates $(x,y)$ corresponds to a vertex $(x,y)$ of $G_{17,2}$. Triangular regions of the same colour correspond to pairs $(x,y)$ lying in the same interval $I_{\ell}$; the lower-left corners of each region lie on the dotted red hyperbola.}
\label{fig:diagram}
\end{figure}

We are now ready to state the main theorem.

\begin{theorem}\label{t:main}
For every integer $n \ge 2$, the shift graph $G_{2^n+1,2}$ contains a unique induced $(n+1)$-vertex-critical subgraph, namely $G_{2^n+1,2}[W]$.
\end{theorem}

\section{Proof of Theorem~\ref{t:main}}
\label{sec:proof-main}

Fix an integer $n\geq 2$, and let us denote the shift graph $G_{2^n+1,2}$ by $G$. Thus, $G$ has chromatic number $\chi(G) = n+1$. Theorem~\ref{t:main} will follow from Theorems~\ref{t:crit-v} and~\ref{t:crit-all} below.

Let $a = (a_1,\dots,a_{2^n+1})$ be a sequence of elements of the set
of all subsets of $[1,n]$, and let $X$ be a subset of the vertex set $V(G)$. Recall that the elements of $X$ are ordered pairs $(x,y)$ of integers satisfying $1 \le x < y \le 2^n + 1$. We say that $a$ is \emph{$X$-good}
if for each $i,j$ such that $1 \leq i < j \leq 2^n+1$,
\begin{equation*}
  a_i \not\subseteq a_j \text{ whenever $(i,j)\in X$.}
\end{equation*}
The subgraph of $G$ induced by the set $X$ is denoted by $G[X]$.

The following proposition characterises induced $n$-colourable
subgraphs of $G$ in terms of $X$-good sequences.

\begin{proposition}\label{p:seq}
  For $X\subseteq V(G)$, $\chi(G[X]) \leq n$ if and only if there
  exists an $X$-good sequence.
\end{proposition}
\begin{proof}
  This is essentially the Erd\H{o}s--Hajnal argument. Suppose that there
  exists an $n$-colouring $c$ of $G[X]$. For $i\in [1,2^n+1]$, let
  \begin{equation*}
    c_i = \Set{c((i,j))}{i < j \leq 2^n+1\text{ and } (i,j)\in X}.
  \end{equation*}
  Note that $c_{2^n+1} = \emptyset$. Furthermore,
  $c_1,\dots,c_{2^n+1}$ is $X$-good, since if $(i,j)\in X$, then
  $c((i,j))\in c_i$ while $c((i,j))\notin c_j$ because of the
  adjacency between $(i,j)$ and the vertex $(j,k)$ for each $k$ such
  that $j < k \leq 2^n+1$. Thus, $c((i,j))$ testifies that
  $c_i \not\subseteq c_j$.

  On the other hand, given an $X$-good sequence
  $(a_1,\dots,a_{2^n+1})$, we colour each vertex $(i,j)\in X$ with an
  arbitrary colour from the nonempty set $a_i\setminus a_j$.
\end{proof}

The following theorem shows that every vertex in $W$ is critical.

\begin{theorem}\label{t:crit-v}
  For $v\in W$, $\chi(G-v) < \chi(G)$.
\end{theorem}
\begin{proof}
  Let $v = (i,j)\in W$ --- say, $i,j\in I_r$ for
  some $r$, $0\leq r\leq n$. We construct a $(V(G)-v)$-good sequence
  which will show that $\chi(G-v) = n$. Let $A =
  \Setx{1,\dots,n-r}$. We set $a_i = a_j = A$. Since $i \geq 2^r$, we
  can position all proper supersets of $A$ at the beginning of the subsequence
  $a_1,\dots,a_{i-1}$. To eventually obtain a $(V(G)-v)$-good sequence,
  we do it in such a way that the size of the sets of the sequence is
  non-increasing. Since $j \leq 2^n - 2^{n-r} + 2$, all proper subsets of $A$
  can be placed at the end of $a_{j+1},\dots,a_{2^n+1}$ (in a similar
  non-increasing way). Finally, we put the subsets of $[1,n]$ which do
  not yet appear in the sequence in the places left free so far. The
  only required property is that, again, the size of these sets is
  non-increasing. It is not hard to check that
  for any ordering satisfying the above requirements, we obtain a
  $(V(G)-v)$-good sequence, which implies that $\chi(G-v) = n$.
\end{proof}

To prove Theorem~\ref{t:main}, it just remains to show that $W$
induces an $(n+1)$-chromatic subgraph of $G$.

\begin{theorem}\label{t:crit-all}
  $\chi(G[W]) = \chi(G) = n+1.$
\end{theorem}
\begin{proof}
  By contradiction. Consider a $W$-good sequence
  $a = (a_1,\dots,a_{2^n+1})$. We first transform the sequence to another
  $W$-good sequence using the following procedure. If there is an
  index $r$ such that some proper subset $a'_r$ of $a_r$ is not
  included among $a_{r+1},\dots,a_{2^n+1}$, let $s$ be the largest
  such index and replace $a_s$ (at position $s$) in the sequence with
  $a'_s$. We prove that the resulting sequence is $W$-good. Suppose
  not. Since $a'_s\subseteq a_s$, there cannot be $i < s$ such that
  $(i,s)\in W$ and $a_i\subseteq a'_s$ as this would contradict the
  $W$-goodness of the original sequence. Thus, let us consider $j > s$
  such that $(s,j)\in W$. In this case, $a'_s$ cannot be a subset of
  $a_j$, because it would be a proper subset (by the choice of $s$),
  so by the maximality of $s$, it would be included among
  $a_{j+1},\dots,a_{2^n+1}$, contradicting the choice of $s$. Thus,
  the resulting sequence is $W$-good, as claimed.

  Relabel the new element $a'_s$ as $a_s$, and repeat the above step
  until this is no longer possible. Clearly, the
  algorithm only takes a finite number of steps and terminates with a
  $W$-good sequence $a_1,\dots,a_{2^n+1}$ with the following property:
  \begin{itemize}
  \item[] For each $i\in [1,2^n+1]$ and every proper subset $b \subset a_i$, there is $j \in [i+1,2^n+1]$ such that $a_j = b$.
  \end{itemize}
 
  Observe that, as a
  consequence, $a_{2^n+1} = \emptyset$. Furthermore, for each
  $j\in[1,2^n+1]$,
  \begin{equation}
    \label{eq:1}
    \text{if }\size{a_j} = k, \text{ then } j \leq 2^n - 2^k + 2,
  \end{equation}
  as there has to be space to the right of $a_j$ for $2^k-1$ proper
  subsets of $a_j$.

  Since the number of subsets of $[1,n]$ is $2^n$ while the length of the sequence
  $a$ is $2^n+1$, one of the subsets needs to occur multiple times in $a$. In particular, there exists a pair $i,j\in [1,2^n+1]$
  such that:
  \begin{itemize} 
    \item[] $i < j$ and $a_i \subseteq a_j$; moreover, we choose this pair
  such that $i$ is as large as possible and, subject to this condition, $j$ is as small as possible.
  \end{itemize}
  Let us write $k = \size{a_j}$. 

  By the maximality of $i$, the subsequence $(a_{i+1}, \dots, a_{2^n+1})$ of $a$ consists of pairwise distinct subsets of $[1,n]$. In addition, it is easy to see that by the choice of $i$ and $j$, the only superset of $a_j$ in this subsequence is $a_j$ itself. The fact that the remaining $2^{n-k}-1$ supersets are missing implies an upper bound on the length of the subsequence:
  \begin{equation*}
    2^n - i + 1 \leq 2^n - (2^{n-k}-1),
  \end{equation*}
  from which we find $2^{n-k} \leq i$. In combination with~\eqref{eq:1}, we obtain
  \begin{equation}
    \label{eq:6}
    2^{n-k} \leq i < j \leq 2^n - 2^k + 2.
  \end{equation}
  This just says that $i,j\in I_{n-k}$. As a consequence, $(i,j)\in
  W$, which contradicts the $W$-goodness of $a$ and concludes the proof of Theorem~\ref{t:crit-all}.
\end{proof}

\bibliographystyle{plain}

\end{document}